\documentclass[12p]{amsart}

\usepackage{amssymb,latexsym}     
\usepackage[utf8]{inputenc}
\usepackage[english]{babel}

\usepackage{mathrsfs}
\usepackage{amsmath,amsthm}
\usepackage{graphicx}
\usepackage{tikz-cd}
\usepackage{adjustbox}
\usepackage{tikz}
\usetikzlibrary{matrix}
\usepackage{caption}

\newtheorem{Lemma}{Lemma}[section]
\newtheorem{Proposition}[Lemma]{Proposition}
\newtheorem{Theorem}[Lemma]{Theorem}

\theoremstyle{definition}

\newtheorem{Definition}[Lemma]{Definition}

\usepackage[lmargin=1in,rmargin=1in,bmargin=0.5in,tmargin=1in]{geometry}
\begin{document}

\title{On two letter identities in Lie rings}

\author{Boris Baranov}
\address{Laboratory of Continuous Mathematical Education (School 564 of St. Petersburg), nab. Obvodnogo kanala 143, Saint Petersburg, Russia.}\email{BBBOOORRRIIISSS@mail.ru}

\author{Sergei O. Ivanov}
\address{Laboratory of Modern algebra and Applications, St. Petersburg State University, 14th Line, 29b,
Saint Petersburg, 199178 Russia.} \email{ivanov.s.o.1986@gmail.com}

\author{Savelii Novikov}
\address{Laboratory of Continuous Mathematical Education (School 564 of St. Petersburg), nab. Obvodnogo kanala 143, Saint Petersburg, Russia.}
\email{novikov.savelii00@gmail.com}
\begin{abstract} Let $L=L(a,b)$ be a free Lie ring on two letters $a,b.$  We investigate the kernel $I$ of the map $L\oplus L \to L $ given by $(A,B)\mapsto [A,a]+[B,b].$ Any homogeneous element of $L$ of degree $\geq 2$ can be presented as $[A,a]+[B,b].$ Then $I$ measures how far such a presentation from being unique. Elements of $I$  can be interpreted as identities $[A(a,b),a]=[B(a,b),b]$ in Lie rings.  The kernel $I$ can be decomposed into a direct sum $I=\bigoplus_{n,m} I_{n,m},$ where elements of $I_{n,m}$ correspond to identities on commutators of weight $n+m,$ where the letter $a$ occurs $n$ times and the letter $b$ occurs $m$ times. We give a full description of $I_{2,m};$ describe the rank of $I_{3,m};$ and present a concrete non-trivial element in $I_{3,3n}$ for $n\geq 1.$   
\end{abstract}
\maketitle
\section*{Introduction}
It is easy to check that the following identity is satisfied in any Lie ring (=Lie algebra over $\mathbb Z$)
\begin{equation}\label{1e}
[a,b,b,a]=[a,b,a,b],
\end{equation} 
where $[x_1,\dots,x_n]$ is the left-normed bracket of elements $x_1,\dots, x_n$ defined by recursion  $[x_1,\dots,x_{n}]:=[[x_1,\dots, x_{n-1}],x_n].$  We denote by $[a,_i b]$ the Engel brackets of $a,b:$
$$[a,_0b]=a, \hspace{1cm} [a,_{i+1}b]=[[a,_ib],b].$$ For example, $[a,_3b]=[a,b,b,b].$
In \cite{Qbad} the second author together with Roman Mikhailov generalized the identity \eqref{1e} as follows
\begin{equation}\label{eq_Qbad_identities}
[[a,_{2n}b],a]= \left[\ \sum_{i=0}^{n-1}(-1)^i [[a,_{2n-1-i} b],[a,_{i}b]]\ ,\  b \ \right],
\end{equation}
where $n\geq 1$. This identity is crucial in their proof that the wedge of two circles $S^1\vee S^1$ is a $\mathbb Q$-bad space in sense of Bousfield-Kan. Note that  the letter $a$ occurs twice in each commutator of this identity and the letter $b$ occurs $2n$ times. Moreover, the identity has the form $[A,a]=[B,b].$ 

We are interested in identities of the form: 
\begin{equation}\label{2}
[A(a,b),a]=[B(a,b),b],
\end{equation}
where $A$ and $B$ are some expressions on  letters $a$ and $b$. These identities can be interpreted as an equalities in the free Lie ring $L=L(a,b).$ Note that a  description of all identities of such kind would give a full description of the intersection $[L,a] \cap [L,b].$
Consider a $\mathbb Z$-linear map
\[
\Theta:L \oplus L \to L,
\] 
\[
\Theta(A,B)=[A,a]+[B,b].
\]
Then the problem of describing identities of type \eqref{2} can be formalised as the problem of describing $$I:={\rm Ker}(\Theta).$$ 
Any homogeneous element of $L$ of degree $\geq 2$ can be presented as $[A,a]+[B,b].$ So $I$ measures how far this presentation from being unique.    
The problem of describing of $I$ is different from the problem formulated on the formal language of identities, because $A,B$ here are not just formal expressions but they are elements of the free Lie ring. For example, the identity $[[b,b],a]=[[a,a],b]$ is not interesting for us because $([b,b],-[a,a])=(0,0)$ in $I.$ This work is devoted to the study of $I.$

The Lie ring $L$ has a natural grading by the weight of a commutator: $L=\bigoplus_{n\geq 1} L_n.$ Moreover, $L_n=\bigoplus_{k+m=n}L_{k,m}$, where $L_{k,m}\subseteq L_{k+m}$ is an abelian group generated by multiple commutators with $k$ letters $a$ and $m$ letters $b.$ We can consider the following restrictions of the map $\Theta$  
$$\Theta_n:L_{n-1}\oplus L_{n-1} \to L_n, \hspace{1cm}\Theta_{k,l}:L_{k-1,l}\oplus L_{k,l-1} \to L_{k,l}$$ and set  $I_n={\rm Ker}(\Theta_n)$ and $I_{k,l}={\rm Ker}(\Theta_{k,l}).$ It is easy to check that 
$$I=\bigoplus_{n\geq 1} I_n, \hspace{1cm} I_n=\bigoplus_{k+l=n} I_{k,l}.$$
The main results of the paper are the full description of $I_{2,n};$ the description of the rank of the free abelian group $I_{3,n};$ and the description of a concrete series of elements from $I_{3,3n}$ for any $n\geq 1.$ 

The rank of a free abelian group $X$ is called ``dimension of $X$'' in this paper and it is denoted by ${\rm dim}\, X.$  It well known that the dimension of $L_n$ can be computed by the Necklace polynomial 
$$ {\rm dim}\, L_n=\frac{1}{n} \sum_{d\mid n} \mu \left(\frac{n}{d}\right)2^d,$$
where $\mu$ is the Mobius function. 
$$\begin{array}{c|c|c|c|c|c|c|c|c|c|c|c|c|c|}
n &1 &2 &3 & 4 & 5 & 6 & 7 & 8 & 9 & 10 & 11 & 12 & 13 \\ \hline
{\rm dim}\, L_n  & 2 & 1 & 2 & 3 & 6 & 9 & 18 & 30 & 56 & 99 & 186 & 335 & 630 \\ \hline
\end{array}$$
Since the map $\Theta_n$ is an epimorphsm, we obtain 
$${\rm dim}\, I_n=2 \cdot {\rm dim}\ L_{n-1}-{\rm dim}\ L_n.$$
We prove the following 
$$
\dim\, I_{2,m}=
\begin{cases} 
0,\, \mbox{if } m \mbox{ is odd} \\
1,\, \mbox{if } m \mbox{ is even} 
\end{cases}, \hspace{1cm} \dim\, I_{3,m}=\left\lfloor \frac{m+1}{2} \right\rfloor - \left\lfloor \frac{m-1}{3} \right\rfloor - 1.
$$
(Proposition \ref{dimI2n}, Proposition \ref{dimI3m}). For $n\leq 13$ we obtain the following table for dimensions. 
$$\begin{array}{c|c|c|c|c|c|c|c|c|c|c|c|c|c|}
n &2 &3 & 4 & 5 & 6 & 7 & 8 & 9 & 10 & 11 & 12 & 13 \\ \hline
{\rm dim}\, I_{2,n-2}  &  0 & 0 & 1 & 0 & 1 & 0 & 1 & 0 & 1 & 0 & 1 & 0  \\ \hline
{\rm dim}\, I_{3,n-3}  &  0 & 0 & 0 & 0 & 1 & 0 & 1 & 1 & 1 & 1 & 2 & 1  \\ \hline
{\rm dim}\, I_n  &  3 & 0 & 1 & 0 & 3 & 0 & 6 & 4 & 13 & 12 & 37 & 40  \\ \hline
\end{array}$$

The computation of ${\rm dim}\, I_{2,m}$ shows that there are no any other elements in $I_{2,m}$ except those that come from the identities \eqref{eq_Qbad_identities}. In particular, all non-trivial identities corresponding to elements of $I_{2,n}$ have even weight. 

Note that $I_n=0$ for odd $n<9.$ This can be interpreted as the fact that there is no a non-trivial identity of the type $[A,a]=[B,b]$ on two letters of odd weight lesser than $9.$ However, we have found a non-trivial identity of this type of weight $9$ (Theorem \ref{th_I36}). If we set 
$$C_n=[a,_nb],$$
then the following identity of weight $9$ holds in any Lie ring
\begin{equation}\label{eq_36}
\big[2[C_5,C_1]+5[C_4,C_2],\ a\ \big]=\big[  2[C_4,C_1,C_0]+3[C_3,C_2,C_0]-2[C_3,C_1,C_1]+[C_2,C_1,C_2], \ b \ \big].
\end{equation} 

The main result of this paper is a concrete series of identities that correspond to non-trivial elements in $I_{3,3n}$ that generalise the identity \eqref{eq_36} (Theorem \ref{3aId}). Namely, for any $n\geq 1$, the following identity is satisfied in $L_{3,3n}$.
\[
\left[\sum_{k=0}^{\left\lfloor\frac{n+1}{2}\right\rfloor}(-1)^{n+1}\alpha_{n+1-k, k}[C_{2n+1-k}, C_{n+k-1}]\ ,\ a\right] = \left[\sum_{i=0}^{n}\sum_{j=0}^{\left\lfloor\frac{i}{2}\right\rfloor}(-1)^{i+1}\alpha_{i-j, j}[C_{n+i-j}, C_{n+j-1},C_{n-i}]\ ,\ b\right],
\]
where $\alpha_{0,0}=1$ and $\alpha_{i,j}=2\binom{i+j-1}{j}+\binom{i+j-2}{j-1}-\binom{i+j-2}{j-2}-2\binom{i+j-1}{j-2}$ for $i,j\geq 0$ and $(i,j)\ne (0,0)$. For $n=2$ we obtain the identity \eqref{eq_36}. For $n=1$ we obtain the identity
$$
[\ 3[a,b,b, [a,b]]+ 2[a,b,b,a]\ ,\ a]\ =\ [\ -[a,b,a,[a,b]] + 2[a,b,b,a,a]\ ,\ b]
$$
that holds in any Lie ring. 

\section{Identities corresponding to elements in $I_{2,m}$}

\begin{Definition} If $w$ is a Lyndon word, we denote by  $[w]$ the corresponding element of the Lyndon-Shirshov basis of the free Lie algebra $L$ (see \cite{Reutenauer}). If $w$ is a letter, then $[w]=w.$ If $w$ is not a letter then $w$ has a standard factorisation $w=uv$ and $[w]$ is defined by recursion $[w]=[[u],[v]].$ 
For example,  $[a]=a$ and $[ab^n]=[[ab^{n-1}],b]=C_n$.
\end{Definition}
\begin{Lemma}\label{basL2n}
The following set is a basis of $L_{2,n}$ with $n\in\mathbb{N}$.
\[
\left\{ \,[C_k, C_l] \;|\; k>l,\: k+l=n \;\;k,l,m\in\mathbb{N} \,\right\}.
\]
\end{Lemma}
\begin{proof}
The intersection of the Lyndon-Shirshov basis with $L_{k,m}$ is a basis of $L_{k,m}.$  The basis of $L_{2,m}$ consists of commutators of Lyndon words with $2$ letters ``$a$'' and $m$ letters ``$b$''. 
\[
[ab^lab^k]=[[ab^l],[ab^k]]=-[[ab^l],[ab^k]]=-[C_k,C_l].
\]
Word $ab^lab^k$ is a Lyndon word only when $k>l$. The assertion follows. 
\end{proof}

\begin{Lemma}\label{dim2n}
For any $n\in\mathbb{N}$ the following is satisfied:
\[
\dim L_{2,n}=\left\lceil \frac{n}{2} \right\rceil = \left\lfloor \frac{n+1}{2}\right\rfloor.
\]
\end{Lemma}

\begin{proof}
Consider the basis from lemma \ref{basL2n}. Hence $L_{2,n}=
\langle \,\{ [C_{n_1},C_{n_2}] \, | \, n_1,n_2\in\mathbb{N}_0, n_1 > n_2 \mbox{ and }  {n_1}+{n_2}=n \} \, \rangle$. Total number of words with $2$ letters $a$ and $n$ letters $b$ starting with $a$ is $n+1$. However, in our case $n_1>n_2$. Hence for odd $n$ number of such commutators is $\frac{n+1}{2}$ and for even $n$ it is $\frac{n}{2}$. 
\end{proof}

\begin{Proposition}\label{dimI2n}
For any $m\in\mathbb{N}$ we have
\[
\dim\, I_{2,m}=
\begin{cases} 
0,\, \mbox{if } m \mbox{ is odd} \\
1,\, \mbox{if } m \mbox{ is even}. 
\end{cases}
\]
\end{Proposition}

\begin{proof}
By definition, $I_{2,m}={\rm Ker}\, \Theta_{2,m}$. Then $\dim I_{2,m}=\dim (\ker\Theta_{2,m})=\dim (L_{1,m} \oplus L_{2,m-1})-\dim (\mathrm{Im} \, \Theta_{2,m})= \dim (L_{1,m} \oplus L_{2,m-1})-\dim L_{2,m}=\dim L_{1,m}+\dim L_{2,m-1}-\dim L_{2,m}=1+\dim L_{2,m-1}-\dim L_{2,m}=1+\left\lfloor \frac{m}{2}\right\rfloor -
\left\lfloor \frac{m+1}{2}\right\rfloor$ (see lemma \ref{dim2n}). Let $m$ be even, then $\dim I_{2,m}=1+ \frac{m}{2} - \left\lfloor \frac{m}{2} -\frac{1}{2}\right\rfloor=1+\frac{m}{2}-\frac{m}{2}=1$. Consider the case of odd $m$. Then $\dim I_{2,m} = 1 + \left\lceil \frac{m-1}{2} \right\rceil - \frac{m+1}{2}=1+\frac{m}{2} -\frac{1}{2}-\frac{m}{2}-\frac{1}{2}=0$.

\end{proof}

\begin{Theorem}
For any $m\in\mathbb{N}$ the following is satisfied
\[
I_{2,m}=
\begin{cases}
0,\, \mbox{if } m \mbox{ is odd} \\
\langle\, (\, C_m ,\, \sum_{i=1}^{\frac{m}{2}}(-1)^i [C_{m-i}, C_{i-1}]\, ) \, \rangle,\, \mbox{if }  m \mbox{ is even}.
\end{cases}
\]	 
\end{Theorem}

\begin{proof}
Triviality of the kernel for odd $m$ can be easily proven using lemma \ref{dim2n}. Consider the case when $m$ is even. Then basis of $L_{1,m}$ consists of one element $[ab^m]$, i.e. $L_{1,m}=\left\{ \alpha[ab^m] \; | \; \alpha\in\mathbb{Z} \right\}$. Basis of $L_{2,m-1}$ consists of $\left\lfloor \frac{m}{2}\right\rfloor$ elements (according to lemma \ref{dim2n}). Because $m$ is even $\left\lfloor \frac{m}{2}\right\rfloor=\frac{m}{2}$. Hence the following equality is true
\[
L_{2,m-1}=\left\{ \alpha_1[aab^{m-1}]+\alpha_2[abab^{m-2}]+\dots+\alpha_{\frac{m}{2}}[ab^{\frac{m}{2}-1}ab^{m-\frac{m}{2}}] \; | \; \alpha_1,\alpha_2,\ldots,\alpha_{\frac{m}{2}}\in\mathbb{Z} \right\}.
\] 
By definition, $I_{2,m}=\ker \Theta_{2,m}$. We can apply map $\Theta_{2,m}$ to arbitrary element of $L_{1,m}\oplus L_{2,m-1}$ that is expressed as basis elements and equate the obtained to zero. Jacobi identity implies the following
\[
[[ab^nab^m], b]=[[ab^{n+1}],[ab^m]]+[ab^nab^{m+1}].
\] 
We can use this equality to transform an image of an element from $L_{2,m-1}$. 
\[
\Theta_{2,m}\left(\alpha[ab^m], \sum_{i=1}^{\frac{m}{2}}\alpha_i [ab^{i-1}ab^{m-i}]\right)=\alpha[[ab^m],a] + \sum_{i=1}^{\frac{m}{2}}\alpha_i [[ab^{i-1}ab^{m-i}], b]=
\]
\[
=-\alpha[a,[ab^m]]+\sum_{i=1}^{\frac{m}{2}}\alpha_i \left([[ab^{i}],[ab^{m-i}]]+[ab^{i-1}ab^{m-i+1}]\right)=
\]
For all $i\not = \frac{m}{2}$ commutator $[[ab^{i}],[ab^{m-i}]]=[ab^{i}ab^{m-i}]$. Then the sum can be rewritten as follows
\[
= -\alpha[aab^m]+\alpha_1[abab^{m-1}]+\alpha_1[aab^m]+\alpha_2[ab^2ab^{m-2}]+\alpha_2[abab^{m-1}]+\dots+\alpha_{i-1}[ab^{i-1}ab^{m-i+1}]+
\]
\[
+\alpha_{i-1}[[ab^{i-2}ab^{m-i+2}]]+\alpha_i[ab^{i}ab^{m-i}]+\alpha_i[ab^{i-1}ab^{m-i+1}]+\alpha_{i+1}[ab^{i+1}ab^{m-i+1}]+\alpha_{i+1}[ab^{i}ab^{m-i}]+\dots+
\]
\[
+\alpha_{\frac{m}{2}}[[ab^{\frac{m}{2}}],[ab^{m-\frac{m}{2}}]]+\alpha_{\frac{m}{2}}[ab^{\frac{m}{2}-1}ab^{m-\frac{m}{2}+1}]=0.
\]
Last but one element of sum is equals to $0$ because $\alpha_{\frac{m}{2}}[[ab^{\frac{m}{2}}],[ab^{m-\frac{m}{2}}]]=
\alpha_{\frac{m}{2}}[[ab^{\frac{m}{2}}],[ab^{\frac{m}{2}}]]=0$. It is easy to see that for equality we need such coefficients $\alpha$ and $\alpha_i$ that terms of the sum will be reduced. Let $\alpha=1$, hence $\alpha_1=1$ because we need 
$-\alpha[aab^m]$ and $\alpha_1[aab^m]$ to be reduced. Other coefficients can be obtained similarly. Commutators with  coefficients $\alpha_i$ and $\alpha_{i+1}$ will be reduced. Hence $\ker \Theta_{2,m}$ is generated by element $[ab^m]=C_m$ and sum 
$[aab^{m-1}]-[abab^{m-2}]+\dots\mp[ab^{\frac{m}{2}-1}ab^{m-\frac{m}{2}+1}]\pm[ab^{\frac{m}{2}}ab^{m-\frac{m}{2}}]=\sum_{i=1}^{\frac{m}{2}}(-1)^{i+1}[ab^{i-1}ab^{m-i}]=$ 
$=\sum_{i=1}^{\frac{m}{2}}(-1)^{i+1}[[ab^{i-1}],[ab^{m-i}]]=\sum_{i=1}^{\frac{m}{2}}(-1)^{i}[[ab^{m-i}],[ab^{i-1}]]=\sum_{i=1}^{\frac{m}{2}}(-1)^{i}[C_{m-i},C_{i-1}] $   
\end{proof}

\section{Identities corresponding to elements in $I_{3,m}$}
\begin{Lemma}\label{basisL3n}
For any $n\in\mathbb{N}$ the following set is a basis of $L_{3,n}$.
\[
\left\{ \,[C_k, C_l, C_m] \;|\; k>l,\,k\geq m, \, k+l+m=n \;\;k,l,m\in\mathbb{N}_0 \,\right\}.
\]
\end{Lemma}

\begin{proof}
Lyndon words commutators of length $n+3$ with $3$ letters ``a'' and $n$ letters ``b'' construct the basis of $L_{3, n}$. It is easy to prove that $ab^{i}ab^{j}ab^{t}$ is a Lyndon word if and only if $i\leq j$ and $i < t$, where $i,j,t\in\mathbb{N}_0$. Consider two cases:
\begin{enumerate}
\item $j < t$ then $[ab^iab^jab^t]=[[ab^i],[ab^jab^t]]=[[ab^i],[[ab^j],[ab^t]]]=[[ab^t],[ab^j],[ab^i]]=[C_t,C_j,C_i]$. Take $t=k, \: j=l, \: i=m$ then $k > l, \: k>m$ and $l\geq m$.
\item $j\geq t$ then $[ab^iab^jab^t]=[[ab^iab^j],[ab^t]]=[[ab^i],[ab^j],[ab^t]]=[C_i ,C_j,C_t]=-[C_j ,C_i,C_t]$.
Take $j=k, \: u=l, \: t=m$ then $k \geq m, \: l\leq k$ and $m>l$. Hence $k\geq m>l$, so $k>l$.
\end{enumerate}
If we unite conditions of both cases, we get $k>l$ and $k\geq m$ for arbitrary $k,l,m\in\mathbb{N}_0$. 
\end{proof}

\subsection{Generalized identity with three letters ``a''}

\begin{Lemma}\label{recur}
For expression $\alpha_{i,j}=2\binom{i+j-1}{j}+\binom{i+j-2}{j-1}-\binom{i+j-2}{j-2}-2\binom{i+j-1}{j-2}$, where $i,j \in \mathbb{N}_0$ the following conditions are satisfied: 
\begin{enumerate}
\item[1)] $\alpha_{i-1,j}+\alpha_{i,j-1}=\alpha_{i,j}$, when $i \ne j$ and $j\ne 0$
\item[2)] $\alpha_{i,i-1}=\alpha_{i,i}$, when $i\ge2$
\item[3)] $\alpha_{i,o}=2$, when $i\geq1$
\end{enumerate}
\end{Lemma}

\begin{proof}
We can use the recurrence relation for binomial coefficient to prove the first condition:
\[
\alpha_{i-1,j}+\alpha_{i,j-1}= 2\binom{i+j-2}{j}+\binom{i+j-3}{j-1}-\binom{i+j-3}{j-2}-2\binom{i+j-2}{j-2} +
\]
\[ +\:  2\binom{i+j-2}{j-1}+\binom{i+j-3}{j-2}-\binom{i+j-3}{j-3}-2\binom{i+j-2}{j-3} 
= 2 \left(\binom{i+j-2}{j} + \binom{i+j-2}{j-1}\right) +
\]
\[+\left(\binom{i+j-3}{j-1} + \binom{i+j-3}{j-2}\right) -  \left(\binom{i+j-3}{j-2}+ \binom{i+j-3}{j-3}\right) - 2\left(\binom{i+j-2}{j-2} + \binom{i+j-2}{j-3}\right) = 
\]
\[
=2\binom{i+j-1}{j}+\binom{i+j-2}{j-1}-\binom{i+j-2}{j-2}-2\binom{i+j-1}{j-2}=\alpha_{i,j}.
\]
To prove the second condition we can substitute $j=i$ into $\alpha_{i, j}$ and express each term using recurrence relation for binomial coefficient:
\[
\alpha_{i,i}=2\binom{2i-1}{i}+\binom{2i-2}{i-1}-\binom{2i-2}{i-2}-2\binom{2i-1}{i-2}=
\]
\[
=\underline{2\binom{2i-2}{i-1}}+2\binom{2i-2}{i}+\underline{\binom{2i-3}{i-2}}+\binom{2i-3}{i-1}-\underline{\binom{2i-3}{i-3}}-\binom{2i-3}{i-2}-\underline{2\binom{2i-2}{i-3}}-2\binom{2i-2}{i-2}=
\]
\[
=\alpha_{i,i-1}+2\binom{2i-2}{i}+\binom{2i-3}{i-1}-\binom{2i-3}{i-2}-2\binom{2i-2}{i-2}.
\]
All we need to prove now is that $2\binom{2i-2}{i}+\binom{2i-3}{i-1}-\binom{2i-3}{i-2}-2\binom{2i-2}{i-2}=0$. Using symmetric property of binomial coefficient, i.e. $\binom{n}{k}=\binom{n}{n-k}$, all terms will be reduced:
\[
2\binom{2i-2}{i}+\binom{2i-3}{i-1}-\binom{2i-3}{i-2}-2\binom{2i-2}{i-2}=\binom{2i-3}{i-2}+2\binom{2i-2}{i-2}-\binom{2i-3}{i-2}-2\binom{2i-2}{i-2}=0.
\]
Consider the case, when $j=0$. We need to mention that for $k<0$ binomial coefficient $\binom{n}{k}=0$. Then the expression will be as follows.
\[
\alpha_{i,0}= 2\binom{i-1}{0}+\binom{i-2}{-1}-\binom{i-2}{-2}-2\binom{i-1}{-2}=2\binom{i-1}{0}=2.
\]
\end{proof}
\begin{Theorem}\label{3aId}
For any $n\in\mathbb{N}$, the following identity is satisfied in $L_{3,3n}$.
\[
\left[\sum_{k=0}^{\left\lfloor\frac{n+1}{2}\right\rfloor}(-1)^{n+1}\alpha_{n+1-k, k}[C_{2n+1-k}, C_{n+k-1}], a\right] = \left[\sum_{i=0}^{n}\sum_{j=0}^{\left\lfloor\frac{i}{2}\right\rfloor}(-1)^{i+1}\alpha_{i-j, j}[C_{n+i-j}, C_{n+j-1},C_{n-i}], b\right],
\]
where $\alpha_{0,0}=1$ and $\alpha_{i,j}=2\binom{i+j-1}{j}+\binom{i+j-2}{j-1}-\binom{i+j-2}{j-2}-2\binom{i+j-1}{j-2}$, where $i,j \in\mathbb{N}$.
\end{Theorem}
\begin{proof} Consider $n, k\in\mathbb{N}$. Lets prove the following identity for any $k\leq n$:
\[
\left[\sum_{i=0}^{k}\sum_{j=0}^{\left\lfloor\frac{i}{2}\right\rfloor}(-1)^{i+1}\alpha_{i-j, j}[C_{n+i-j}, C_{n+j-1},C_{n-i}], b\right]=\sum_{t=0}^{\left\lfloor\frac{k+1}{2}\right\rfloor}(-1)^{k+1}\alpha_{k+1-t, t}[C_{n+k+1-t}, C_{n-1+t}, C_{n-k}].
\]
Denote the left part as $\omega_k$ and the right part as $\theta_k$. We will prove this equality using mathematical induction with variable $k$. \\
\par \fbox{1} We can expand the sum and use lemma \ref{recur}
\[\omega_{1}=\left[-\alpha_{0,0}[C_{n},C_{n-1},C_{n}]+\alpha_{1,0}[C_{n+1},C_{n-1},C_{n-1}],b\right]=-1[C_{n},C_{n-1},C_{n},b]+2[C_{n+1},C_{n-1},C_{n-1}, b]=
\]
\[
=-1[C_{n+1},C_{n-1},C_{n}]+1[C_{n+1},C_{n},C_{n-1}]+2[C_{n+1},C_{n-1},C_{n}]+2[C_{n+2},C_{n-1},C_{n-1}]+2[C_{n+1},C_{n-1},C_{n-1}]=
\]
\[
=2[C_{n+2},C_{n-1},C_{n-1}]+3[C_{n},C_{n-1},C_{n}]=\alpha_{2,0}[C_{n+2},C_{n-1},C_{n-1}]+\alpha_{1,1}[C_{n},C_{n-1},C_{n}]=\theta_1
\]
\par\fbox{2} We need to prove that $\omega_{k}=\theta_{k}$ implies $\omega_{k+1}=\theta_{k+1}$.
We can expand $\omega_{k+1}$ as a sum of $\omega_{k}$ and the last element of the first sum in $\omega_{k+1}$:
\[
\omega_{k+1}=\underbrace{\omega_{k}}_{\theta_{k}}+\left[\sum_{j=0}^{\left\lfloor\frac{k+1}{2}\right\rfloor}(-1)^{k+2}\alpha_{k+1-j, j}[C_{n+k+1-j}, C_{n+j-1},C_{n-k-1}],b\right]
\]
By expanding the commutator as follows $[C_{n+k+1-j}, C_{n+j-1},C_{n-k-1},b]=[C_{n+k+2-j}, C_{n+j-1},C_{n-k-1},b]+[C_{n+k+1-j}, C_{n+j},C_{n-k-1},b]+[C_{n+k+1-j}, C_{n+j-1},C_{n-k},b] $, we can express the second term as three different sums. One of them will be reduced by $\theta_{k}$.
\[
\omega_{k+1}=\sum_{t=0}^{\left\lfloor\frac{k+1}{2}\right\rfloor}(-1)^{k+1}\alpha_{k+1-t, t}[C_{n+k+1-t}, C_{n-1+t}, C_{n-k}]+\sum_{j=0}^{\left\lfloor\frac{k+1}{2}\right\rfloor}(-1)^{k+2}\alpha_{k+1-j, j}[C_{n+k+2-j}, C_{n+j-1},C_{n-k-1}]+
\]
\[
+\sum_{j=0}^{\left\lfloor\frac{k+1}{2}\right\rfloor}(-1)^{k+2}\alpha_{k+1-j, j}[C_{n+k+1-j}, C_{n+j},C_{n-k-1}]+\sum_{j=0}^{\left\lfloor\frac{k+1}{2}\right\rfloor}(-1)^{k+2}\alpha_{k+1-j, j}[C_{n+k+1-j}, C_{n+j-1},C_{n-k}]=
\]
\[
=\sum_{j=0}^{\left\lfloor\frac{k+1}{2}\right\rfloor}(-1)^{k+2}\alpha_{k+1-j, j}[C_{n+k+2-j}, C_{n+j-1},C_{n-k-1}]+\sum_{j=0}^{\left\lfloor\frac{k+1}{2}\right\rfloor}(-1)^{k+2}\alpha_{k+1-j, j}[C_{n+k+1-j}, C_{n+j},C_{n-k-1}]
\]
Denote commutators in sums as $a_j$ and $b_j$ correspondingly. We can show that for any $j\geq 0$ it is satisfied that $a_{j+1}=b_j$.
\[
a_{j+1}=[C_{n+k+2-j-1}, C_{n+j+1-1},C_{n-k-1}]=[C_{n+k+1-j}, C_{n+j},C_{n-k-1}]=b_j.
\]
Because of that, the expression can be written as a sum of $a_0$, $b_{\left\lfloor\frac{k+1}{2}\right\rfloor}$ with corresponding coefficients and one sum on index $j$ with summed coefficients. 
\[
\omega_{k+1}=\sum_{j=0}^{\left\lfloor\frac{k+1}{2}\right\rfloor-1}(-1)^{k+2}(\alpha_{k-j, j+1} + \alpha_{k+1-j, j})[C_{n+k+1-j}, C_{n+j},C_{n-k-1}] +(-1)^{k+2}\alpha_{k+1,0}[C_{n+k+2}, C_{n-1}, C_{n-k-1}]-
\]
\[
+(-1)^{k+2}\alpha_{k+1-\left\lfloor\frac{k+1}{2}\right\rfloor, \left\lfloor\frac{k+1}{2}\right\rfloor}[C_{n+k+1-\left\lfloor\frac{k+1}{2}\right\rfloor},C_{n+\left\lfloor\frac{k+1}{2}\right\rfloor},C_{n-k-1}].
\]
Coefficients of commutators, in obtained sum on index $j$, can be transformed using the first case of lemma \ref{recur}. Also, we can change index of sum by subtracting $1$ from it. Coefficient of $a_0$ can be rewritten using the third case of lemma \ref{recur}. Then $b_{\left\lfloor\frac{k+1}{2}\right\rfloor}$ will become a zero element of sum on index $j$. 
\[
\omega_{k+1}=\sum_{j=1}^{\left\lfloor\frac{k+1}{2}\right\rfloor}(-1)^{k+2}\alpha_{k+2-j, j}[C_{n+k+2-j}, C_{n-1+j},C_{n-k-1}] +(-1)^{k+2}\alpha_{k+2,0}[C_{n+k+2}, C_{n-1}, C_{n-k-1}]+
\]
\[
+(-1)^{k+2}\alpha_{k+1-\left\lfloor\frac{k+1}{2}\right\rfloor, \left\lfloor\frac{k+1}{2}\right\rfloor}[C_{n+k+1-\left\lfloor\frac{k+1}{2}\right\rfloor},C_{n+\left\lfloor\frac{k+1}{2}\right\rfloor},C_{n-k-1}]=
\]
\[
=\sum_{j=0}^{\left\lfloor\frac{k+1}{2}\right\rfloor}(-1)^{k+2}\alpha_{k+2-j, j}[C_{n+k+2-j}, C_{n-1+j},C_{n-k-1}] + (-1)^{k+2}\alpha_{k+1-\left\lfloor\frac{k+1}{2}\right\rfloor, \left\lfloor\frac{k+1}{2}\right\rfloor}b_{\left\lfloor\frac{k+1}{2}\right\rfloor}.
\]
\noindent Consider two cases:

\noindent \textbf{1)} $k$ is odd. Then $\left\lfloor\frac{k+2}{2}\right\rfloor=\left\lfloor\frac{k+1}{2} + \frac{1}{2}\right\rfloor=\frac{k+1}{2}+\left\lfloor\frac{1}{2}\right\rfloor=\frac{k+1}{2}=\left\lfloor\frac{k+1}{2}\right\rfloor$, hence $b_{\left\lfloor\frac{k+1}{2}\right\rfloor}=0$. It is true because $b_{\left\lfloor\frac{k+1}{2}\right\rfloor}=[C_{n+k+1-\frac{k+1}{2}},C_{n+\frac{k+1}{2}},C_{n-k-1}]=[0,C_{n-k-1}]=0$. Consequently $\omega_{k+1}$ can be expressed as follows.
\[
\omega_{k+1}=\sum_{j=0}^{\left\lfloor\frac{k+2}{2}\right\rfloor}(-1)^{k+2}\alpha_{k+2-j, j}[C_{n+k+2-j}, C_{n-1+j},C_{n-k-1}]= \theta_{k+1}.
\] 

\noindent\textbf{2)} $k$ is even. Then $\left\lfloor\frac{k+2}{2}\right\rfloor=\frac{k+2}{2}=\left\lfloor\frac{k}{2}\right\rfloor + 1=\left\lfloor\frac{k+1}{2}\right\rfloor + 1$, hence $\alpha_{k+2-\left\lfloor\frac{k+2}{2}\right\rfloor, \left\lfloor\frac{k+2}{2}\right\rfloor}=\alpha_{\left\lfloor\frac{k+1}{2}\right\rfloor + 1, \left\lfloor\frac{k+1}{2}\right\rfloor}$ because of the second case of lemma \ref{recur}. It is important to mention that $\left\lfloor\frac{k+1}{2}\right\rfloor = \left\lfloor\frac{k}{2} + \frac{1}{2}\right\rfloor=\frac{k}{2}$. Hence $\alpha_{k+1-\left\lfloor\frac{k+1}{2}\right\rfloor, \left\lfloor\frac{k+1}{2}\right\rfloor} = \alpha_{\frac{k}{2}+1, \left\lfloor\frac{k+1}{2}\right\rfloor}=\alpha_{\left\lfloor\frac{k+1}{2}\right\rfloor + 1, \left\lfloor\frac{k+1}{2}\right\rfloor}$. We can express $\theta_{k+1}$ as a sum up to $\left\lfloor\frac{k+2}{2}\right\rfloor - 1 = \left\lfloor\frac{k+1}{2}\right\rfloor$ and the last addendum with $j=\left\lfloor\frac{k+2}{2}\right\rfloor=\left\lfloor\frac{k+1}{2}\right\rfloor + 1$:
\[
\theta_{k+1}=\sum_{j=0}^{\left\lfloor\frac{k+1}{2}\right\rfloor}(-1)^{k+2}\alpha_{k+2-j, j}[C_{n+k+2-j}, C_{n-1+j},C_{n-k-1}] +
\] 
\[
+ (-1)^{k+2}\alpha_{k+2-\left\lfloor\frac{k+2}{2}\right\rfloor, \left\lfloor\frac{k+2}{2}\right\rfloor}[C_{n+k+2-\left\lfloor\frac{k+1}{2}\right\rfloor- 1},C_{n-1+\left\lfloor\frac{k+1}{2}\right\rfloor + 1},C_{n-k-1}]=  
\]
\[
=\sum_{j=0}^{\left\lfloor\frac{k+1}{2}\right\rfloor}(-1)^{k+2}\alpha_{k+2-j, j}[C_{n+k+2-j}, C_{n-1+j},C_{n-k-1}] + (-1)^{k+2}\alpha_{k+1-\left\lfloor\frac{k+1}{2}\right\rfloor, \left\lfloor\frac{k+1}{2}\right\rfloor}b_{\left\lfloor\frac{k+1}{2}\right\rfloor}=\omega_{k+1}
\]

\ 
\\ \\  Consequently, the identity above is satisfied for any $n,k\in\mathbb{N}$, such that $k\leq n$. If we substitute $k=n$, we will get the original identity.
\end{proof}
\subsection{Additional results}

\begin{Lemma}\label{razn3n}
For any $n\in\mathbb{N}$ the following is satisfied:
\[
\dim L_{3,n}-\dim L_{3,n-1}=\left\lfloor \frac{n-1}{3} \right\rfloor + 1.
\]
\end{Lemma}

\begin{proof}
To calculate this expression, we need to count all Lyndon words of form $ab^{n_1}ab^{n_2}ab^{n_3}$, where $n_1,n_2,n_3\in\mathbb{N}_0  \mbox{ and } {n_1}+{n_2}+{n_3}=n$. Let $n_1=i$ and $n_2=j$, hence $n_3=n-i-j$. 
As it was mentioned before, $ab^{n_1}ab^{n_2}ab^{n_3}$ is a Lyndon word if and only if $n_1\leq n_2$ and $n_1 < n_3$, where $n_1,n_2,n_3\in\mathbb{N}$. We can portray integer points that satisfy these conditions on coordinate plane by drawing plots of functions $y=x$ and $y=n-2x$.
\vspace{-0.6cm}
\begin{center}
\includegraphics[scale=0.94]{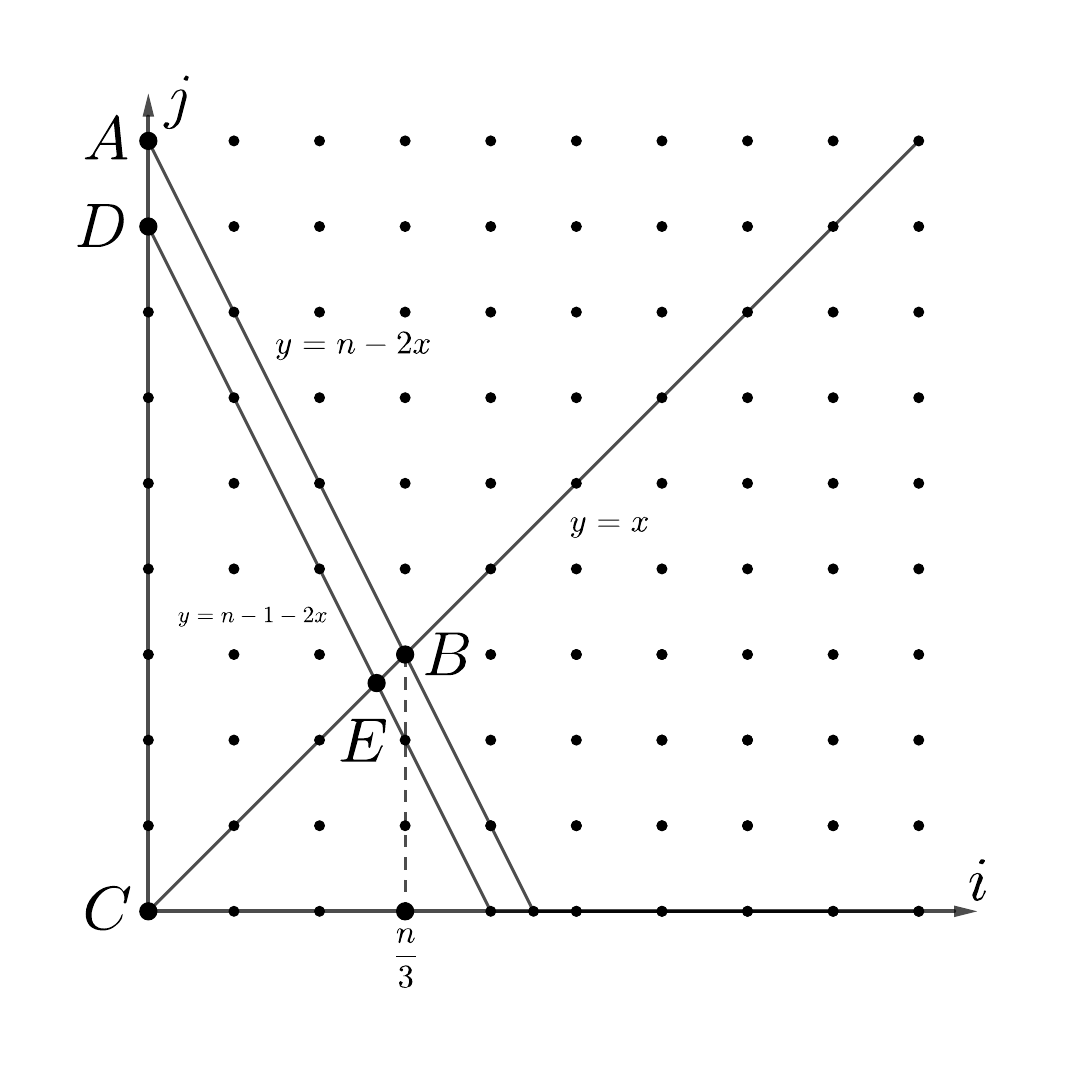}
\end{center}
\vspace{-0.5cm}
Abscissa of functions intersection point is $\frac{n}{3}$. $ab^{i}ab^{j}ab^{n-i-j}$ is a Lyndon word if point $(i,j)$ belongs to $	\bigtriangleup ABC$ (without point on the line $y=n-2x$). Then $\dim L_{3,n}$ equals to number of integer points in $	\bigtriangleup DEC$. Hence $\dim L_{3,n}-\dim L_{3,n-1}$ equals to number of integer points on segment $DE$, i.e. $\left\lfloor \frac{n-1}{3} \right\rfloor + 1$.
\end{proof}

\begin{Proposition}\label{dimI3m}

For any $m\in\mathbb{N}$ the following is satisfied:
\[
\dim I_{3,m}=\left\lceil \frac{m}{2} \right\rceil - \left\lfloor \frac{m-1}{3} \right\rfloor - 1.
\]
\end{Proposition}

\begin{proof}
By definition, $I_{3,m}=\ker\Theta_{3,m}$. Hence, according to lemmas \ref{dimI2n} and \ref{razn3n}, $\dim I_{3,m}=\dim \ker\Theta_{3,m}=\dim (L_{2,m} \oplus L_{3,m-1})-\dim L_{3,m} = \dim L_{2,m} - (\dim L_{3,m} - \dim L_{3,m-1})=\left\lceil \frac{m}{2} \right\rceil - \left\lfloor \frac{m-1}{3} \right\rfloor - 1$.
\end{proof}

\begin{Lemma}\label{razl}
For $k>l$, $k\geq m$ the following is satisfied:
\[
[C_k, C_l, C_m, b]=
\begin{cases}
[C_{k+1}, C_l, C_m] + [C_k, C_{l+1}, C_m] + [C_k, C_l, C_{m+1}], \mbox{if } k>l+1,\: k\geq m+1 \\
[C_{k+1}, C_l, C_m] + [C_k, C_l, C_{m+1}],  \mbox{if } k=l+1,\: k\geq m+1 \\
2[C_{k+1}, C_l, C_m] - [C_{k+1}, C_{l+1}, C_{m-1}], \mbox{if } k=l+1,\: k=m \\
2[C_{k+1}, C_l, C_m] + [C_k, C_{l+1}, C_m] - [C_{k+1}, C_k, C_l], \mbox{if } k>l+1,\: k=m.
\end{cases}
\]
\end{Lemma}

\begin{proof}
It is easy to rewrite the expression in the first case using Jacobi identity:
\[
[C_k, C_l, C_m, b]=[C_k,C_l,b,C_m]+[C_k,C_l,[C_m,b]]
=[C_{k+1}, C_l, C_m] + [C_k, C_{l+1}, C_m] + [C_k, C_l, C_{m+1}].
\]
Second case:
\[
[C_{k}, C_l, C_m, b]=[C_{l+2}, C_l, C_m] + [C_{l+1}, C_{l+1}, C_m] + [C_{l+1}, C_l, C_{m+1}]=[C_{k+1}, C_l, C_m] +  [C_{k}, C_l, C_{m+1}].
\]
Third case:
\[
[C_k, C_l, C_m, b]=[C_{k+1}, C_l, C_{m}] + [C_{k}, C_l, C_{m+1}]=[C_{k+1}, C_l, C_m]+[C_k,C_{m+1},C_l]+[C_{m+1},C_l,C_k]=
\]
\[
=2[C_{k+1},C_l,C_m]-[C_{k+1},C_{l+1},C_{m-1}].
\]
Fourth case:
\[
[C_k, C_l, C_m, b]
=[C_{k+1}, C_l, C_m] + [C_k, C_{l+1}, C_m] + [C_k, C_l, C_{m+1}]= [C_{k+1}, C_l, C_m] + [C_k, C_{l+1}, C_m] + 
\]
\[
+[C_k,C_{m+1},C_l]-[C_l,C_{m+1},k]=[C_{k+1}, C_l, C_m] + [C_k, C_{l+1}, C_m] + [C_k,C_{m+1},C_l]-[C_l,C_{m+1},k]=
\]
\[
=2[C_{k+1}, C_l, C_m] + [C_k, C_{l+1}, C_m] - [C_{k+1}, C_k, C_l].
\]
\end{proof}
\begin{Theorem}\label{I33}
The kernel of $\Theta_{3,3}$ is generated by the following element: 
\[
(3[C_2, C_1]+ 2[C_3,C_0], [C_1,C_0,C_1] - 2[C_2,C_0,C_0] ).
\] 
\end{Theorem}

\begin{proof}
According to lemma \ref{dimI3m} $\dim I_{3,3}=\left\lceil \frac{3}{2} \right\rceil - \left\lfloor \frac{2}{3} \right\rfloor - 1=1$. Consequently, we have to provide only one identity to describe the whole $I_{3,3}$. Substitute $n=1$ into identity from theorem \ref{3aId}:
\[
[\alpha_{2,0}[C_3,C_0] + \alpha_{1,1}[C_2,C_1],a]=[-\alpha_{0,0}[C_1,C_0,C_1]+\alpha_{1,0}[C_2,C_0,C_0],b].
\]
By definition of $\alpha_{i,j}$, $\alpha_{2,0}=2, \: \alpha_{1,1}=3,\: \alpha_{0,0}=1$ and $\alpha_{1,0}=2$. We can move right part of the equality to the left side and it will become an image of the element from $L_{2,3} \oplus L_{3,2}$:
\[\hspace{-0.5cm}
[2[C_3,C_0] + 3[C_2, C_1],a] + [[C_1,C_0,C_1] - 2[C_2,C_0,C_0],b] = \Theta_{3,3}(2[C_3,C_0]+3[C_2, C_1] , [C_1,C_0,C_1] - 2[C_2,C_0,C_0])= 0
\] As a result, we obtained the element that generates all identities in $L_{3,3}$ that is equivalent to description of $I_{3,3}$.
\end{proof}

\begin{Theorem}\label{th_I36}
The abelian group $I_{3,6}$ is generated by the following element  
\[(-2[C_5,C_1]-5[C_4,C_2]\ ,\  2[C_4,C_1,C_0]+3[C_3,C_2,C_0]-2[C_3,C_1,C_1]+[C_2,C_1,C_2] )
\] 
\end{Theorem}

\begin{proof}
Similarly to proof of the theorem \ref{I33}. $\dim I_{3,6}=\left\lceil \frac{6}{2} \right\rceil - \left\lfloor \frac{5}{3} \right\rfloor - 1=1$. Substitute $n=2$ into identity from theorem \ref{3aId}:
\[
[-\alpha_{3,0}[C_5,C_1]-\alpha_{2,1}[C_4,C_2], a]=[-\alpha_{0,0}[C_2,C_1,C_2]+\alpha_{1,0}[C_3,C_1,C_1]-\alpha_{2,0}[C_4,C_1,C_0]-\alpha_{1,1}[C_3,C_2,C_0],b].
\]
Coefficients will be $\alpha_{3,0}=2, \: \alpha_{2,1}=5, \: \alpha_{0,0}=1, \: \alpha_{1,0}=2, \: \alpha_{2,0}=2$ and $\alpha_{1,1}=3$. Again, we've found an element of $L_{2,6} \oplus L_{3,5}$ that generates all possible identities. Coefficients in the right part will be multiplied by $-1$ because of moving to the left side. 
\end{proof}

\end{document}